\documentclass{amsart}
\usepackage{amsmath}
\usepackage{amsfonts}
\usepackage{amssymb}
\usepackage{url}
\usepackage{epsfig}
\usepackage[pdftex]{hyperref}

\usepackage{epstopdf}
\usepackage{color}
\usepackage{bbm}

\definecolor{blue}{RGB}{197,77,87}
\definecolor{green}{RGB}{200,244,99}
\definecolor{red}{RGB}{78,205,196}

\definecolor{grey}{RGB}{197,77,87}

\newtheorem{theorem}{Theorem}
  \newtheorem{lemma}[theorem]{Lemma}

   \newtheorem{corollary}[theorem]{Corollary}

  \theoremstyle{definition}
  \newtheorem{definition}[theorem]{Definition}

\newcommand{\RR}{\mathbb R}
\newcommand{\QQ}{\mathbb Q}

\newcommand{\cI}{\mathcal I}

\newcommand{\cB}{\mathcal B}

\newcommand{\cM}{\mathcal M}

\DeclareMathOperator{\val}{val}
\DeclareMathOperator{\init}{in}

\DeclareMathOperator{\Trop}{Trop}

\DeclareMathOperator{\rank}{rank}

\DeclareMathOperator{\spann}{span}

\DeclareMathOperator{\cone}{cone}

\date{\today}
 
\begin{document}

 \title[Algebraic Matroids and Set-theoretic Realizability of Tropical Varieties]{Algebraic Matroids and Set-theoretic Realizability of Tropical Varieties}
  \author{Josephine Yu}
 \address{School of Mathematics, Georgia Institute of Technology,
        Atlanta GA, USA}
\email {jyu@math.gatech.edu}

 \begin{abstract}
To each prime ideal in a polynomial ring over a field we associate an algebraic matroid and show that it is preserved under tropicalization.  This gives a necessary condition for a tropical variety to be set-theoretically realizable from a prime ideal.  We also show that there are infinitely many Bergman fans that are not set-theoretically realizable as the tropicalization of any ideal.
 \end{abstract}

 \maketitle

Let $P$ be a prime ideal in the polynomial ring $K[x_1,\dots,x_n]$ over
a field $K$.   Algebraic independence over
$K$ gives a matroid structure on the set $\{[x_1],\dots,[x_n]\} \subset
K[x_1,\dots,x_n]/P$ where $[x_i]$ denotes the coset of $x_i$
in the quotient ring. Matroids that arise this way are said to be {\em
  algebraic} over $K$.\footnote{This definition is equivalent to the
  more common definition of algebraic matroids using algebraic
  independence over $K$ among elements in an extension field $L
  \supset K$; one can see this by
considering the $K$-algebra homomorphism from $K[x_1,\dots,x_n]$ to
$L$ that sends $x_i$'s to the matroid elements in $L$.}  
 Vector matroids are algebraic
because the prime ideal $P$ can be taken to be the ideal generated by linear relations
among the vectors. The class of algebraic matroids is
closed under taking minors, but it is not known whether it
is closed under taking duals~\cite{Oxley}.

Non-algebraic matroids exist.  Ingleton and Main showed that the Vamos matroid,
which is a rank $4$ self-dual matroid on an $8$ element set, is not algebraic
over any field~\cite{IngletonMain}.  Lindstr\"om constructed an infinite class
of non-algebraic matroids of rank $3$ all of whose proper minors are
algebraic~\cite{Lindstrom87} and an infinite class of non-algebraic
matroids, one for each rank $\geq 4$, such that no member of this
class is a minor of another~\cite{Lindstrom88}. 

A subset $S\subset \{1,\dots,n\}$ is independent in the algebraic
matroid of a prime ideal $P$ if and only if $P
\cap K[x_s : s\in S] = \{ 0 \}$.  Note that after fixing a generating
set of $P$, replacing $K$ by any extension field
does not change this last condition as it can be checked using
Gr\"obner basis computations over $K$.  Moreover, if a matroid is algebraic over
$K$, then it is algebraic over any extension of $K$~\cite{Oxley}.  
We may assume that $K$ is algebraically closed.
Let $V(P)$ be the variety in $K^{[n]}$ defined by 
$P$. By Hilbert's Nullstellensatz, a subset
$S \subset \{1,\dots,n\}$ is independent in the algebraic matroid of $P$ if
and only if the projection of $V(P)$ onto the coordinate subspace $K^S$
is onto.

If $P$ contains a monomial, then it must also contain some
$x_i$ because it is prime.  This means that the algebraic matroid of
$P$ contains a loop (a one-element dependent set).   By removing the
loops and the corresponding variables if necessary, we will only deal
with loop-free matroids and monomial-free prime ideals. In this case, instead of considering the variety $V(P)$ in $K^{[n]}$, we can
consider the variety $V(P) \cap (K^*)^{[n]}$, and the algebraic matroid is characterized
by surjectivity of coordinate projections as before.

We wish to apply the same construction of algebraic matroids to tropical varieties.
A {\em tropical variety}\footnote{Elsewhere in the
  literature, such as in the book ~\cite{MaclaganSturmfels}, the term  {\em tropical
    variety} is used only for tropicalizations of ideals; however here
  we use it for fans that may or may not arise from an ideal.} in $\QQ^n$ is a pure weighted balanced rational
polyhedral fan.\footnote{We could have used $\RR$ instead of $\QQ$,
  but in the proof of Lemma~\ref{lem:same} we will need to use a larger
  extension field whose value group is $\RR$.} A positive integer,
called weight, is assigned to (interior points of) each maximal cone of the tropical
variety, and the weights satisfy a {\em balancing condition} along
each ridge, which means that the weighted sum of primitive integer
vectors pointing from the ridge into each incident facet lies in the
linear span of the ridge~\cite[\S3.3]{MaclaganSturmfels}.  We consider two
pure weighted balanced rational polyhedral fan to be the same tropical
variety if they have the same ground set and the weights agree on a
dense subset.  We will not dwell on the details about the weights as the main focus here is the projections of the ground set.

\begin{definition}
 The {\em independence
  complex } $\cI(T)$ of a tropical variety $T \subset \QQ^n$ is the collection of subsets $S \subseteq
\{1,\dots,n\}$ such that the image of the projection of $T$ to the coordinate
subspace $\QQ^S$ has dimension~$|S|$.  
\end{definition}
Since projection of tropical
varieties are again tropical varieties~\cite{JensenYu}, by the balancing condition,
the condition that the image of $T$ is full-dimensional is equivalent to the
condition that the image of $T$ is all of $\QQ^S$.

For an ideal $J \subset  K[x_1,\dots,x_n]$, the {\em
  tropicalization} of $J$ (or the tropicalization of the variety $V(J)$) is 
$$
\Trop(J) = \{-w \in \QQ^n : \init_w(J) \text{ does not contain a
  monomial} \},
$$
which has a polyhedral fan structure derived, for instance,
from the Gr\"obner fan of a homogenization of $J$. If $J$ is equidimensional, then $\Trop(J)$ is a tropical
variety, where the weight at a generic point $w\in \Trop(J)$ is
defined as the sum of multiplicities of the initial ideal $\init_w(J)$
along its monomial-free minimal associated
primes~\cite{MaclaganSturmfels}.  For two ideal $I_1,I_2$ of the
same dimension, we have $\Trop(I_1 \cap I_2) = \Trop(I_1) \cup
\Trop(I_2)$ as sets, and the weight of a point in $\Trop(I_1 \cap I_2)$ is the sum of
its weights in $\Trop(I_1)$ and $\Trop(I_2)$, where the weight is $0$
if the point is not the tropical variety.
An important question in tropical geometry is to determine if a
tropical variety with its weights is {\em realizable}, that
is, if it is the tropicalization of an ideal.  After fixing a generating set of $J$, the tropicalization remains the same if $K$ is replaced by any extension, as
the tropicalization can be computed using Gr\"obner bases
over $K$.  We say that a tropical variety is {\em set-theoretically realizable} if it is the ground set of a tropicalization of an ideal, i.e.\ we disregard the weights.  

The following lemma shows that tropicalization preserves algebraic matroids of prime ideals.   The {\em independence
  complex of a matroid} is the collection of independent sets in the matroid.
\begin{lemma}
\label{lem:same}
For any monomial-free prime ideal $P$ in $K[x_1,\dots, x_n]$, the
independence complex $\cI(\Trop(P))$ coincides with the independence
complex of the algebraic matroid of $P$.
\end{lemma}

\noindent In particular, for a tropical variety to be realizable as the tropicalization of a prime ideal, it is necessary that its independence complex forms an algebraic matroid.

\begin{proof}
Let $\widetilde K$ be an algebraically closed extension field of $K$
with a valuation $$\val: {\widetilde K}^* \rightarrow \QQ \text{ with
} \val(K^*) = 0,$$ whose value
group is equal to $\QQ$.  More concretely, if $K$ has characteristic~$0$, then
we can take $\widetilde K$ to be the field of Puiseux series over $\overline{K}$, and
if $K$ has positive characteristic, then we can take $\widetilde K$ to be
the field of generalized power series or Hahn series over $\overline{K}$ as in~\cite{Kedlaya}.
As noted above, replacing $K$ by the extension field $\widetilde K$ changes neither
the algebraic matroid nor the tropicalization of $P$.

By the Fundamental Theorem of Tropical
Geometry~\cite{MaclaganSturmfels}, we have
$$
\Trop(P) = \{(\val(x_1),\dots,\val(x_n)) : (x_1,\dots,x_n)  \in
V_{\widetilde K}(P) \cap 
({\widetilde K}^*)^n\}.
$$
Since taking valuation commutes with coordinate projections, the
result follows.
\end{proof}

We will now show that that for
every loop-free matroid $M$, algebraic or not, there is a tropical variety whose
independence complex is $M$.
 For any loop-free matroid $M$ of rank $r$ on $n$ elements, one can
construct a tropical variety of dimension $r$ in $\QQ^n$, called the
Bergman fan $\cB(M)$ of $M$ as follows~\cite{ArdilaKlivans}. Using the
{\em min} convention in tropical geometry, the Bergman fan $\cB(M)$
is the union of cones of the form 
$$
\cone \{\chi_{F_1}, \dots, \chi_{F_{k}}\} + \QQ (1,1,\dots,1)
$$
where $F_1 \subsetneq \cdots \subsetneq F_{k} \subsetneq M$ is a chain
of flats of the
matroid and the vector $\chi_{F} \in \{0,1\}^{[n]}$ denotes the indicator function of
$F$.  The Bergman fan of any matroid, with all weights equal to~$1$,
forms a tropical variety.  The balancing condition can be proved using
the covering partition property of flats of a matroid~\cite[\S2.2]{HuhThesis}.  If the matroid $M$ is representable as a vector matroid over
$K$, then the Bergman fan $\cB(M)$ is the tropicalization of the
ideal generated by linear relations among the vectors.   In this case, by
Lemma~\ref{lem:same}, the independence complex of $\cB(M)$ forms the
independence complex of $M$.  In fact, this is true for all matroids.

\begin{lemma}
\label{lem:Bergman}
For any loop-free matroid $M$, the independence complex $ \cI(\cB(M))$
of the Bergman fan $\cB(M)$ coincides with the independence complex of $M$.  
\end{lemma}

\begin{proof}
Suppose $S = \{s_1,\dots,s_k\} \subseteq \{1,\dots,n\}$.  If $S$ is
independent in the matroid $M$, then we get a chain of flats
$$\spann\{s_1\} \subsetneq \spann\{s_1,s_2\} \subsetneq \cdots \subsetneq
\spann\{s_1,\dots,s_k\} \text{ in } M.$$ The projection of the cone spanned by their indicator functions has full dimension in $\QQ^S$.  

For the converse, suppose $S$ is dependent in $M$.  For any
chain of flats $$F_1 \subsetneq \cdots \subsetneq F_{k} \subsetneq M,$$
we get a chain of flats in the matroid on $S$ obtained from $M$ by
restriction:
$$F_1 \cap S \subset \cdots \subset F_{k}
\cap S \subset S.$$
Since $\rank(S) < |S|$, the chain above contains at most $|S|-1$
different flats in $S$.  Since the projection of $\chi_F$ is the same as the
projection of $\chi_{F \cap S}$ onto the $S$-coordinates, the
projection of no cone in $\cB(M)$ can have full dimension in $\QQ^S$.
\end{proof}

\begin{theorem}
\label{thm:main}
If the Bergman fan of a matroid $M$ is set-theoretically realizable over a field $K$, then $M$ is algebraic over $K$.
\end{theorem}

\begin{proof}
For any matroid $M$, the only
weights on the Bergman fan $\cB(M)$ that satisfy the balancing condition are those
that give the same weight to all maximal
cones~\cite[Theorem~38]{HuhThesis}.  This can be proved using 
shellability of $\cB(M)$ and the flat partition property of matroids.
It follows that $\cB(M)$ does not contain proper tropical subvarieties of the same dimension.  If $\cB(M)$ is
the tropicalization of an ideal $J$, then all the top dimensional
associated primes of~$J$ have tropicalization equal to $\cB(M)$. Thus
if $\cB(M)$ is realizable, then it is realizable by a prime ideal.
The result then follows from Lemmas~\ref{lem:same} and~\ref{lem:Bergman}. 
\end{proof}

There has been speculations that perhaps every Bergman fan is realizable over every algebraically closed field as the tropicalization of an ideal if the weights are allowed to be scaled up. The existence of non-algebraic matroids and Theorem~\ref{thm:main} tell us that this is not the case.  
\begin{corollary}
\label{cor:Huh}
There exist infinitely many Bergman fans that are not realizable over any field, with respect to any weight.
\end{corollary}

We end with some of open problems.
\begin{enumerate}
\item A rational homology class in a complete variety is called {\em prime} if some positive multiple of it is the class of an irreducible subvariety.  
Is every Bergman fan, considered as a rational homology class in the permutohedral variety, a limit of prime classes up to numerical equivalence?   This question is due to June Huh~\cite[\S4.3]{HuhThesis}.  
\item Is the converse of Theorem~\ref{thm:main} true?
\item Which tropical varieties have matroidal independence complexes? What is the right combinatorial substitute for the irreducibility condition for varieties?
\end{enumerate}

\bigskip

\begin{center}
{\bf Acknowledgments}
\end{center}

I thank Matt Baker, Dustin Cartwright, Anders Jensen, Diane Maclagan, and the referees for helpful discussions and comments.  This work was supported by the NSF-DMS grant \#1101289.

\bibliographystyle{amsalpha}
\bibliography{algMatroid}

 \end{document}